\documentclass[11pt]{article}

\usepackage{amsfonts,amsthm,amsmath}
\usepackage{graphics}
\usepackage{lpic}

\newcommand{\len}{\mathrm{len}}
\newcommand{\rank}{\mathrm{rank}}

\newcommand{\hh}{\mathbb{H}^3}

\newtheorem{lem}{Lemma}
\newtheorem{prop}{Proposition}
\newtheorem{thm}{Theorem}
\newtheorem{cor}{Corollary}

\title{A note on the uniqueness of minimal length carrier graphs}
\author{Michael Siler}

\begin{document}
\maketitle
\begin{abstract}
\noindent We show that minimal length carrier graphs are not unique, but if $M$
is in a large class of hyperbolic 3-manifolds, including the geometrically
finite ones, then $M$ has only finitely many minimal length carrier graphs and
no two of them are homotopic. As a corollary, we obtain a new proof that the
isometry group of a geometrically finite 3-manifold is finite.
\end{abstract}

Let $M$ be a hyperbolic 3-manifold. A \emph{carrier graph} for $M$ is a finite
graph $X$ along with a map $f\!:\!X\to M$ such that
$f_{\ast}\!:\!\pi_1(X)\to\pi_1(M)$ is surjective. We will assume that
$\pi_1(M)$ is finitely generated, and it is then clear that a carrier graph for
$M$ exists. We will only consider carrier graphs with
$\rank(\pi_1(X))=\rank(\pi_1(M))$. The length of an edge $e$ of $X$,
$\len_f(e)$, is the length of the path $f|_e$ (we assume $f$ takes edges to
rectifiable paths in $M$), and $\len_f(X)$ is the sum of the lengths of the
edges of $X$. A \emph{minimal length carrier graph} is a carrier graph with
length less than or equal to the length of any other (minimal rank) carrier graph for $M$.
In~\cite{W}, White showed that if $M$ is closed, then it has a minimal length
carrier graph, and in~\cite{S}, it is shown how to extend White's argument to a
much larger class of hyperbolic 3-manifolds. Minimal length carrier graphs have
nice geometric properties and were first used by White~\cite{W} to prove that if
$M$ is closed, then it has a nontrivial loop whose length is bounded above in
terms of nothing more than the rank of $\pi_1(M)$. They have subsequently been
used, for example, to show that rank equals Heegaard genus for large classes of
hyperbolic 3-manifolds in~\cite{So}, \cite{B}, and \cite{NS}.

Following Souto~\cite{So}, we say that two carrier graphs $f\!:\!X\to M$ and
$g\!:\!Y\to M$ are \emph{equivalent} if there exists a homotopy equivalence
$\eta\!:\!X\to Y$ so that $f$ and $g\eta$ are freely homotopic. We will say that
$f$ and $g$ are
\emph{strongly equivalent} if there is a homeomorphism $\eta\!:\!X\to Y$ such
that $f=g\eta$. In this
note, we consider the following two uniqueness questions:
\begin{enumerate}
\item Must any two minimal length carrier graphs for $M$ be strongly equivalent?
\item Must any two carrier graphs which both have minimal length within the
same equivalence class be strongly equivalent?
\end{enumerate}

The answer to both questions is no, according to the examples in
Section~\ref{example}. However, we will prove two weaker
uniqueness results in Section~\ref{theorems}. In order to state the results,
we need one more (very strong) notion of equivalence. Two carrier graphs
$f,g\!:\!X\to M$ are \emph{essentially equivalent} if $f=g\eta$, where
$\eta\!:\!X\to X$ is a homeomorphism that fixes vertices and leaves edges and
their orientations invariant. In other words, $f$ and $g$ are the same except
for reparameterizing the edges. Carrier graphs are \emph{essentially distinct}
if they are not essentially equivalent.

\begin{thm}\label{main_thm} Let $M$ be a hyperbolic 3-manifold and let
$f\!:\!X\to M$ and $g\!:\!X\to M$ be carrier graphs, which either each have
minimal length within their equivalence classes or each have minimal length
globally. If $f$ and $g$ are homotopic, then they are essentially equivalent.
\end{thm}

And although there may be more than one carrier graph of minimal length globally
or within an equivalence class, we show

\begin{thm}\label{finite}
Let $M$ be a hyperbolic 3-manifold that does not have a simply-degenerate,
$\pi_1$-surjective NP-end. Then $M$ has only finitely many essentially distinct
minimal length carrier
graphs, and each equivalence class of carrier graphs can have only finitely many
essentially distinct minimal length representatives.
\end{thm}

The fact that a manifold satisfying the hypotheses of this theorem has a
minimal length carrier is proved in~\cite{S}. In Section~\ref{sec:application},
we will look at the action of the isometry group $\mathrm{Isom}(M)$ of $M$ on
the set of minimal length carrier graphs and derive a new proof of the following
fact:

\begin{cor}\label{finite_isom}
If $M$ does not have a simply-degenerate, $\pi_1$-surjective NP-end, then
$\mathrm{Isom}(M)$ is finite.
\end{cor}

Note that this corollary applies, for example, to all geometrically finite
3-manifolds.

The author would like to thank Peter Shalen for helpful guidance and Benson Farb
for pointing out the above corollary.

\section{Non-uniqueness example}\label{example}

\begin{prop}\label{nonunique}
Let $M$ be a hyperbolic 3-manifold with $\rank(\pi_1(M))=2$. Suppose that $M$
has a minimal length carrier graph $f\!:\!X\to M$ and that $M$ has a fixed-point
free isometry $h$ of finite order not divisible by 3. Then $hf$ is a minimal
length carrier graph not strongly equivalent to $f$.
\end{prop}

\begin{proof}
It is clear that $hf$ is a minimal length carrier graph. Suppose it is strongly
equivalent to $f$. Then there exists a homeomorphism $\eta\!:\!X\to X$ such
that $hf=f\eta$. Note that $\eta$ cannot fix any point $x\in X$, for then
$h$ would fix $f(x)$. According to~\cite{W}, minimal length carrier graphs must
be trivalent. There are only two trivalent graphs of rank 2: one that looks like
a $\theta$ and one that looks like eye-glasses. The eye-glasses graph does not
admit a fixed-point free homeomorphism; so $X$ is the $\theta$ graph. Up to
homotopy, $\eta$ must be the homeomorphism that swaps vertices and cyclically
permutes the edges.

Let $m$ be the order of $h$. Then $h^mf=f\eta^m$, which is equivalent to
$f=f\eta^m$. Since $m$ is not divisible by 3, $\eta^m$ cyclically permutes the
edges of $X$. Hence, $f$ must map each edge to the same image, which contradicts
$f$ being a carrier graph because $f_{\ast}(\pi_1(X))$ would be trivial.
\end{proof}

We can get concrete examples from this proposition. For example, let $M$ be the
figure 8 knot complement. Then $M$ is a two-fold cover of the Gieseking manifold,
hence it has a fixed-point
free isometry $h$ of order 2, and the rank of $\pi_1(M)$ is easily found to be two
(from, say, the Wirtinger presentation); so $M$ has non-unique minimal length
carrier graphs.

We can also get closed examples. Reid~\cite{R}
shows how to produce, for any $p>1$, a closed hyperbolic 3-manifold $M$ with
a regular, cyclic
cover $N$ of degree $p$ such that $\rank(\pi_1(N))=2$. If $p$
is not divisible by 3, then $N$ with its order $p$ deck transformation
satisfies the hypotheses of Proposition~\ref{nonunique} and thus has non-unique
minimal length carrier graphs.

We can take these examples a bit further to get examples of carrier graphs which
are minimal in the same equivalence class but are not strongly
equivalent. Reid's manifolds are formed as follows.
Let $\varphi$ be a pseudo-Anosov homeomorphism of a punctured torus $T$ and let
$M_{\varphi}$ be the mapping torus of $\varphi$. Let $a,b$ be generators of
$\pi_1(T)$. Reid forms a manifold, which we are calling $N$, by taking the obvious
$p$-fold cyclic cover of $M_{\varphi}$ and performing a certain Dehn filling on it.
It is shown that $N$ is a $p$-fold cyclic cover of a manifold obtained from
Dehn filling $M_{\varphi}$ and the preimage of the
 filling torus for $M_{\varphi}$ is the
filling torus of $N$ (in particular, the deck transformations of $N$ leave the
filling torus invariant).
By abuse of notation, we will use $a$ and $b$ to refer to the generators of the
fiber subgroup of $M_{\varphi}$ and its cover and to their images in the filled
manifold $N$. Reid shows that $a$ and $b$ generate $\pi_1(N)$. Let $H$ be the
filling torus of $N$. Then $N\setminus H$ is fiber bundle over $S^1$ with fiber
a compact surface of genus 1 and with 1 boundary component. Choose representatives
$\alpha$ and $\beta$ of $a$ and $b$, respectively, that lie in 
a particular fiber $\Sigma$ of $N\setminus H$. If $h$ is an order $p$
deck-transformation of $N$,
then $h\circ\alpha$ and $h\circ\beta$ are loops in the fiber $h(\Sigma)$ which also
generate $\pi_1(N)$. Notice that there is a submanifold homeomorphic to
$\Sigma\times[0,1]\subset N$ containing $\alpha$, $\beta$, $h\circ\alpha$ and $h\circ\beta$.
The manifold $\Sigma\times[0,1]$ is a genus 2 handlebody and the pairs
$\{\alpha,\beta\}$ and $\{h\circ\alpha,h\circ\beta\}$ each generate its
fundamental group. It is a well-known fact that any two minimal cardinality
generating sets for the fundamental group of a handlebody are Nielsen equivalent;
hence, $h$ preserves the Nielsen equivalence class of the generating pair
$\{a,b\}$.

Let $f\!:\!S^1\vee S^1\to N$ be the carrier graph given by mapping one of the
$S^1$s to $\alpha$ and the other to $\beta$, and let $f'$ be a carrier graph of
minimal length in the equivalence class of $f$. Then $h\circ f'$ has minimal
length in the equivalence class of the graph coming from $h\circ\alpha$ and
$h\circ\beta$. In~\cite{So}, Souto shows how to
associate an equivalence class of carrier graphs to a Nielsen equivalence class
of generators for $\pi_1$ and vice versa. His discussion of this correspondence
implies that since $h$ preserves the Nielsen equivalence class of $\{a,b\}$,
$f'$ and $h\circ f'$ are equivalent. However, Proposition~\ref{nonunique} implies
that these carrier graphs are not strongly equivalent. Hence, minimal length
carrier graphs are not unique even within an equivalence class.

\section{Weaker forms of uniqueness}\label{theorems}

For the proof of Theorem~\ref{main_thm}, we will need a lemma.

\begin{lem}\label{half_triangle}
Let $x$, $y$ and $z$ be distinct points in $\mathbb{H}^n$. Let $x'$ (resp. $y'$) be the
midpoint of the geodesic between $x$ (resp. $y$) and $z$. Then
$d(x',y')\le\frac12d(x,y)$, and equality is achieved exactly when the angle
$\angle xzy$ is 0 or $\pi$.
\end{lem}

\begin{proof}
Let $a=d(x',z)$, $b=d(y',z)$, $c=d(x',y')$ and $\gamma=\angle xzy$. We wish to
show that $2c\le d(x,y)$. This is equivalent to $\cosh(2c)\le\cosh(d(x,y))$. By
the hyperbolic law of cosines,
\begin{eqnarray*}
\cosh(c) &=& \cosh(a)\cosh(b)-\sinh(a)\sinh(b)\cos(\gamma) \\
\cosh d(x,y) &=& \cosh(2a)\cosh(2b)-\sinh(2a)\sinh(2b)\cos(\gamma).
\end{eqnarray*}

Now we just follow our noses: $\cosh(2c)=2\cosh^2(c)-1$, so we need to show
\[
2\cosh^2(c)-1\le \cosh(2a)\cosh(2b)-\sinh(2a)\sinh(2b)\cos(\gamma).
\]
The left side is equal to
\begin{multline*}
2\cosh^2(a)\cosh^2(b)-4\cosh(a)\cosh(b)\sinh(a)\sinh(b)\cos(\gamma)+\\
2\sinh^2(a)\sinh^2(b)\cos^2(\gamma)-1.
\end{multline*}
Notice that
\[
\sinh(2a)\sinh(2b)\cos(\gamma)=4\cosh(a)\cosh(b)\sinh(a)\sinh(b)\cos(\gamma).
\]
So our goal becomes
\[
2\cosh^2(a)\cosh^2(b)+2\sinh^2(a)\sinh^2(b)\cos^2(\gamma)-1\le\cosh(2a)\cosh(2b).
\]
Using the identity $\cosh(2x)=2\cosh^2(x)-1$ and some algebra, one can see that
this is equivalent to
\[
\sinh^2(a)\sinh^2(b)\cos^2(\gamma)+\cosh^2(a)+\cosh^2(b)\le\cosh^2(a)\cosh^2(b)+1.
\]
It suffices to prove this inequality with the assumption that
$\cos^2(\gamma)=1$, or equivalently, $\gamma=0,\pi$. In this case, it is an
equality, which follows from the identity $\cosh^2(b)=\sinh^2(b)+1$.
\end{proof}

\begin{proof}[Proof of Theorem~\ref{main_thm}]
Suppose $f$ and $g$ are homotopic and essentially distinct. Let
$H\!:\!X\times[0,1]\to M$ be a homotopy
from $f$ to $g$. The space $X\times[0,1]$ can be triangulated as follows. Let
$e\subset X$ be an edge. Suppose $e$ has distinct endpoints. Then a
homeomorphism from $e$ to $[0,1]$ can be extended to a homeomorphism from
$e\times[0,1]$ to $[0,1]\times[0,1]$ (sending $e$ to $[0,1]\times\{0\}$). The
latter space has a triangulation with two triangles obtained by splitting the
square along one of its diagonals. This triangulation can be pulled back to a
triangulation for $e\times[0,1]$. If $e$'s endpoints are the same (i.e. $e$ is a
loop), then it can be triangulated in essentially the same way, but with
$e\times\{0\}$ and $e\times\{1\}$ identified. These triangulations can be glued
together in an obvious way to yield a triangulation of $X\times[0,1]$. The map
$H$ can be made simplicially hyperbolic with respect to this triangulation, i.e.
it can be made to send edges to geodesic segments and 2-simplices to geodesic
triangles. We will assume this has been done. Note that this does not
change the ends of the homotopy ($f$ and $g$), since they already have geodesic
edges by virtue of having minimal length.

We now construct a new carrier graph $h\!:\!X\to M$. Let $v$ be a vertex of $X$.
Then $H$ maps $\{v\}\times[0,1]$ to a geodesic arc in $M$. Let $h(v)$ be the
midpoint of that geodesic. Having defined $h$ on the vertices on $X$, we can
define it on the edges. Let $e$ be an edge of $X$ with distinct endpoints $v$
and $w$. Then $e\times[0,1]$ consists of two triangles, which share a common
edge. Let $m$ be the midpoint of the (geodesic) image of that edge under $H$.
There are geodesic arcs $e_1$ and $e_2$ connecting $h(v)$ to $m$ and $m$ to
$h(w)$ and lying within $H(e\times[0,1])$. Let $h$ map $e$ homeomorphically to the
path formed by concatenating $e_1$ and $e_2$. See Figure~\ref{homo_e}. If $e$
has only one endpoint, then $h(e)$ is formed similarly; the picture is the same
as Figure~\ref{homo_e}, except that the left and right edges are identified.
Thus, the map $h$ is essentially the midpoint of the homotopy $H$. It is clear
that $h$ is homotopic to $f$ and $g$, which implies that it is a carrier graph
in the same equivalence class as $f$ and $g$.

\begin{figure}
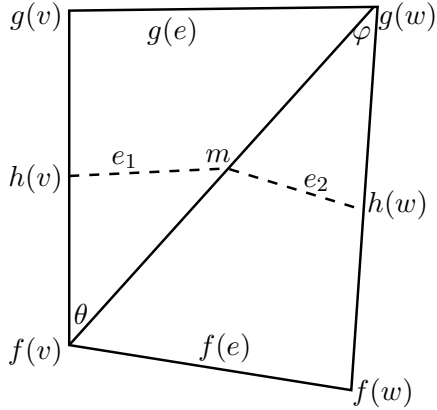

\begin{center}
\begin{lpic}{triangulated_rect(.5)}
\lbl[l]{35,12;$f(e)$}
\lbl[r]{35,96;$g(e)$}
\lbl[r]{-1,57;$h(v)$}
\lbl[l]{80,50;$h(w)$}
\lbl[l]{37,63;$m$}
\lbl[l]{12,62;$e_1$}
\lbl[l]{63,56;$e_2$}
\lbl[r]{-1,11;$f(v)$}
\lbl[l]{76,0;$f(w)$}
\lbl[r]{-1,100;$g(v)$}
\lbl[l]{83,100;$g(w)$}
\lbl[l]{2,21;$\theta$}
\lbl[l]{76,95;$\varphi$}
\end{lpic}
\caption{$H(e\times[0,1])$}\label{homo_e}
\end{center}
\end{figure}

We will show that $\len_h(X)<\frac12(\len_f(X)+\len_g(X))$. Because
$\len_f(X)=\len_g(X)$ (since they both have minimal length), $h$ will be shorter
than both, which will contradict minimality and complete the proof.
Still referring to Figure~\ref{homo_e}, we
can lift to $\hh$ and apply Lemma~\ref{half_triangle} to the left and right
triangles to get
\begin{equation}\label{edge_ineq}
\len_h(e)=\len(e_1)+\len(e_2)\le\frac12\len_g(e)+\frac12\len_f(e)
\end{equation}
with equality if and only if $\theta,\varphi\in\{0,\pi\}$. Summing over all edges,
we get $\len_h(X)\le\frac12(\len_f(X)+\len_g(X))$. In order for this to be a
strict inequality, we need for there to be at least one edge for
which~(\ref{edge_ineq}) is a strict inequality.

For any vertex $v\in X$, let $H_v=H(\{v\}\times[0,1])$. There must be some
vertex $v_0$ such that $H_v$ is not just a point. For otherwise, $f$ and $g$
would agree on the vertices of $X$ and for each edge $e$ of $X$, $f(e)$ and
$g(e)$ would be geodesic segments homotopic relative to their endpoints. Hence,
$f(e)$ would be the same as $g(e)$. If $e$ has distinct endpoints, then it is clear
that on $e$, $f$ and $g$ differ by an orientation preserving homeomorphism. If
$e$ is a loop, then perhaps $f$ and $g$ map $e$ to the same geodesic, but with
opposite orientation. Since $f$ and $g$ are homotopic (via a homotopy that does
not move the vertices), that would imply that the loop $f(e)$ is homotopic to its
inverse. Since $\pi_1(M)$ is torsion-free and simple loops in $X$ map to non-nullhomotopic
loops in $M$, this cannot happen. Therefore, $f$ and $g$ must be essentially
equivalent, which is a contradiction.
Note that if $H_v$ is not a single point, it is a geodesic path.

Pick an edge $e$ with (not necessarily distinct) endpoints $v$ and $w$, such that
$H_v$ is not a single point. If the inequality~(\ref{edge_ineq}) for $e$ is
strict, then we are done. If it is an equality, then the angles $\theta$ and
$\varphi$ must each be either 0 or $\pi$. This implies that the angle between
$H_v$ and $f(e)$ is either 0 or $\pi$.
The vertex $v$ must have some other edge $e'\ne e$ adjacent to it. In~\cite{W},
it is shown that the angle between any two edges sharing a vertex in a minimal length
carrier graph is $2\pi/3$. In particular, the angle between $f(e')$ and $f(e)$ is
$2\pi/3$. Thus, the angle between $f(e')$ and $H_v$ cannot be 0 or $\pi$, and so
for $e'$, the inequality~(\ref{edge_ineq}) must be strict. Hence, we get the
desired contradiction
\[
\len_h(X)<\frac12(\len_f(X)+\len_g(X))=\len_f(X).
\]
\end{proof}

In the examples following Proposition~\ref{nonunique}, we found that minimal
length carrier graphs
were not unique because we can compose them with ambient isometries to get new
carrier graphs. It is perhaps natural to wonder if any two minimal length
carrier graphs are related in this way (up to reparamaterizing their edges). If
this were true, then the well-known fact that for a large class of hyperbolic 3-manifolds
$M$, $\mathrm{Isom}(M)$ is finite, would imply Theorem~\ref{finite}, that there
are only finitely many minimal length carrier graphs. We will prove this theorem
directly using Theorem~\ref{main_thm} via the following proposition and then
prove the finiteness of isometry groups as a corollary in the next section.

\begin{prop}\label{prop}
Let $M$ be a finite volume hyperbolic 3-manifold and let $L>0$. There are only
finitely many carrier graphs which are minimal length within
their equivalence class and have length less than or equal to $L$.
\end{prop}

\begin{proof}
Suppose $M$ has an infinite sequence of carrier graphs $f_i\!:\!X\to M$, each
of minimal length within its equivalence class and each with length less than or
equal to $L$.
Being minimal length implies the graphs are trivalent. There are only finitely
many trivalent graphs of a particular rank; so we may pass to a subsequence and
assume that every $X_i$ is homeomorphic to a particular graph $X$. We will continue
to call this sequence $f_i$. Since the $f_i$ have bounded length and a carrier
graph cannot be contained in a cusp (because this would imply that $\pi_1(M)$
is a quotient of the cusp group $\mathbb{Z}^2$), there is a bound on how deep
into a cusp neighborhood the image of any $f_i$ may penetrate. Hence, the $f_i$
all map into one compact subset of $M$. Additionally, the bound on the length
of the $f_i$ implies that the sequence is equicontinuous.
We can now apply the Arzel\`a-Ascoli theorem to get that a subsequence of $\{f_i\}$
converges uniformly. Therefore, for some large $i$ and $j$, $f_i$ is sufficiently
close to $f_j$ that the two maps must be homotopic. This contradicts
Theorem~\ref{main_thm}.
\end{proof}

\begin{proof}[Proof of Theorem~\ref{finite}]
Let $M$ be a finite volume hyperbolic 3-manifold. Let $\mathcal{C}$ be the set
of minimal length carrier graphs for $M$, and let $L$ be the length of any element
of $\mathcal{C}$. Elements of $\mathcal{C}$ clearly have minimal length within
their equivalence classes. Thus, $\mathcal{C}$ is contained in the set of
carrier graphs which are of minimal
length in their equivalence classes and have length less than or equal to $L$. The
latter set is finite, by virtue of Proposition~\ref{prop}.

Similarly, the set of carrier graphs of minimal length within a particular
equivalence class is seen to be finite by letting $L$ be the length of any
minimal length representative and applying Proposition~\ref{prop} in the same way.
\end{proof}

\section{An application to isometry groups}\label{sec:application}

We will now give a new proof of the previously known result that a hyperbolic
3-manifold that does not have a simply-degenerate, $\pi_1$-surjective NP-end
has finite isometry group. The proof is simple and follows from Theorem~\ref{finite}
and basic facts about minimal length carrier graphs. It requires knowing that
such manifolds have minimal length carrier graphs, which was proved in~\cite{S}.
That proof relies on the proof of the tameness theorem by Agol and Calegari-Gabai.
However, the tameness theorem is not needed for the case in which the 3-manifold
is geometrically finite. Thus, when the following theorem is restricted to
geometrically finite 3-manifolds, its proof is entirely elementary.

\begin{cor}[Corollary 1]
If $M$ does not have a simply-degenerate, $\pi_1$-surjective NP-end, then
$\mathrm{Isom}(M)$ is finite.
\end{cor}

\begin{proof}
Let $\mathcal{C}$ be the set of essential equivalence classes of minimal length
carrier graphs in $M$. By Theorem~1 of~\cite{S}, this set is nonempty, and
by Theorem~\ref{finite}, this set is finite.
It is clear that $\mathrm{Isom}(M)$ acts on $\mathcal{C}$,
which gives a map from $\mathrm{Isom}(M)$ to the finite group of permutations
of $\mathcal{C}$. Let $K$ be the kernel of this map. It suffices to show that
$K$ is finite. Isometries of $M$ that are in $K$ fix a minimal length carrier
graph $f\!:\!X\to M$ up to essential equivalence. In particular, for some vertex $v\in X$,
they fix $f(v)$ and permute the images of the three edges attached to $v$. This gives
a map from $K$ to $S_3$, the permutation group on three elements. An element $h$ of
the kernel of this map would fix $f(v)$ and the three tangent vectors at $f(v)$
corresponding to the three edges of $X$ attached to $v$. Since the angles between
these edges are all $2\pi/3$, these tangent vectors span a plane in the tangent
space. Lifting to $\hh$, $\tilde{h}$ would fix some preimage of $f(v)$ and fix a hyperplane
going through $f(v)$ pointwise (since it is an isometry and fixes the tangent plane).
Hence, $h$ must be the identity map. Thus, $K$ injects into $S_3$, which means
that $K$ is finite. Therefore, $\mathrm{Isom}(M)$ is finite.
\end{proof}

\noindent\textsc{Department of Mathematics, Statistics, and Computer Science (M/C 249),
University of Illinois at Chicago, 851 S. Morgan St., Chicago, IL 60607-7045}\\

\noindent\emph{E-mail address}: \texttt{michael.siler@gmail.com}

\end{document}